\documentclass[11pt]{amsart}

\usepackage{mathrsfs}
\usepackage{bbm}					
\usepackage[parfill]{parskip} 
		
\usepackage{bm}	
\usepackage{amsmath,amssymb,amsthm,eucal,verbatim}
\usepackage{hyperref,amsbsy}
\usepackage{graphicx}
\usepackage{epstopdf}

\usepackage{enumitem}

\usepackage{mathtools}

\advance \topmargin by-\headheight
\advance \topmargin by-\headsep
\evensidemargin \oddsidemargin
\theoremstyle{plain}
\newtheorem{thm}{Theorem}[section]

\newtheorem{cor}{Corollary}[section]
\newtheorem{lem}{Lemma}[section]
\newtheorem*{thm*}{Theorem}
\newtheorem*{lem*}{Lemma}
\newtheorem*{cor*}{Corollary}

\newtheorem*{rem*}{Remark}
\newtheorem{prop}[thm]{Proposition}

\newtheorem*{defn*}{Definition}
\theoremstyle{definition}

\newtheorem*{hypzerochi}{Hypothesis $\mathscr{H}_{0, \chi}$}
\newtheorem*{hypdis}{Hypothesis $\mathscr{D}$}

\newcommand{\be}{\begin{equation}}
\newcommand{\ee}{\end{equation}}
\renewcommand{\a}{\alpha}
\newcommand{\s}{\sigma}
\renewcommand{\b}{\beta}
\renewcommand{\r}{\rho }
\newcommand{\g}{\gamma }
\newcommand{\e}{\epsilon}

\renewcommand{\d}{\delta}   
\newcommand{\z}{\zeta}

\newcommand*{\bigs}[1]{\scalebox{1.2}{\ensuremath#1}}

\makeatletter
\newcommand*{\bigss}[1]{{\hbox{$\left#1\vbox to11\p@{}\right.\n@space$}}}
\makeatother

\newcommand{\sdfrac}[2]{\mbox{\small$\displaystyle\frac{#1}{#2}$}}

\title[A Central Limit Theorem for Linear Combinations]
{A Central Limit Theorem  for Linear Combinations of Logarithms 
of Dirichlet $L$-functions Sampled at the Zeros of the Zeta Function}

\author[F. \c{C}\.{\i}\c{c}ek, S. M. Gonek, and S. J. Kirila]{Fatma \c{C}\.{\i}\c{c}ek, Steven M. Gonek, and Scott J. Kirila}

\address{University of Northern British Columbia\\ Department of Mathematics and Statistics \\ 3333 University Way\\ Prince George, BC\ V2N 4Z9\\ Canada}
\email{cicek@unbc.ca}

\address{Department of Mathematics, University of Rochester, Rochester, NY 14627, USA}
\email{gonek@math.rochester.edu \quad  }

\address{Cincinnati, OH, USA}
\email{kirilasj@gmail.com} 


\begin{document}

\thanks{This article was first published in the Ramanujan Journal, volume 67 (31), 2025 by Springer Nature.}

\maketitle

\begin{abstract}
Let $L(s, \chi_1), \ldots, L(s, \chi_N)$ be primitive Dirichlet $L$-functions different from the Riemann zeta function. Under suitable hypotheses we prove that  any linear combination $a_1\log|L(\rho,\chi_1)|+\dots+a_N\log|L(\rho,\chi_N)|$ has an approximately normal distribution as $T\to \infty $ with mean $0$ and variance $ \tfrac12 \big({a_1}^2+\dots+{a_N}^2\big)\log\log T.$  Here  
  $a_1, a_2, \ldots, a_N \in \mathbb{R}$, and $\rho$ runs over the nontrivial zeros of the zeta function with
$0< \Im\rho \leq T$.
  From this  we  deduce that  the vectors 
$\big(\log|L(\rho,\chi_1)|/\sqrt{ \frac12 \log\log T}, \ldots, \log|L(\rho,\chi_N)|/\sqrt{\frac12 \log\log T}\,\big)$ have approximately an $N$-variate normal distribution whose components are approximately mutually independent as $T\to \infty$.
 We   apply these results to study the proportion of the $\r$ that are zeros or $a$-values of  linear combinations  of the form $c_1 L(\rho, \chi_1)+ \cdots + c_N L(\rho, \chi_N)$ with complex $c_i$ as coefficients.
\end{abstract}

\section{Introduction}\label{intro}

Let $\z(s)$ denote the Riemann zeta function, $L(s, \chi)$ a primitive Dirichlet $L$-function with character $\chi$ modulo $m$, and let $s=\s+it$ be a complex variable.
 The first author recently proved that if the Riemann hypothesis and a standard well-spacing hypothesis about the nontrivial zeros $\r=\frac12+i\g$ of $\z(s)$ 
 are assumed, then the values of $\log|\z(\r+z)|$  with $0<\gamma\leq T$ and $|z|\neq 0$ small are approximately normally distributed  as $T\to\infty$. She proved a similar result for     $\log|\z'(\r)|$ again as $\r$ runs over  the nontrivial zeros of the zeta function, thus providing a different proof with a better error term of a result of Hejhal~\cite{H}.

Our main goal here is to show that, under suitable hypotheses, this normal (or Gaussian) behavior is also exhibited by the real parts of the logarithms of primitive Dirichlet $L$-functions and by  linear combinations of such functions when 
sampled at the nontrivial zeros $\r$ of the zeta function.  We further show  that the logarithms of distinct primitive $L$-functions at the zeros $\r$   are approximately statistically independent. These results have unconditional continuous counterparts  for the functions $\log | L(\frac12+it) |$ (for example, see Selberg~\cite{S} and Hsu and Wong~\cite{Hsu Wong} for Dirichlet $L$-functions).
Apart from being of interest in their own right, they also  reinforce the rule of thumb, well known to experts, that primitive Dirichlet $L$-functions behave independently from one another in a variety of ways. For example, the theorems just mentioned support the idea that  the zeros of the zeta function are not particularly special   as far as  other Dirichlet $L$-functions are concerned.

Before stating our results, we list the three hypotheses we require. 
 First, we   assume the Generalized Riemann hypothesis (GRH) throughout. This  asserts that all nontrivial zeros of all primitive Dirichlet $L$-functions lie on the critical line $\Re s=\frac12$. We also use
\begin{hypdis}
The nontrivial zeros of primitive Dirichlet $L$-functions never coincide with the zeros of the Riemann zeta-function. 
\end{hypdis}
Note that Hypothesis $\mathscr D$ is a special consequence of  the much stronger Grand Simplicity Hypothesis~\cite{RS}, which states that the ordinates of all nontrivial zeros of all primitive Dirichlet  $L$-functions are linearly independent over the rationals.

Next let $\r_{\chi}=\tfrac12+i\g_\chi$ denote a generic nontrivial zero of $L(s, \chi)$, and let $N(T)$ denote the number of zeros $\rho=\frac12+i\g$ of the Riemann zeta-function with $0<\g \leq T$. Our final hypothesis is a `zero-spacing' hypothesis.

\begin{hypzerochi}\label{hypothesis 0}
We have
        \[
      \lim_{\epsilon \to 0^+}  \limsup_{T\to\infty}\frac{1}{N(T)}
        \#\bigg\{0 < \g \leq T: \text{there exists a }\,  \r_\chi \text{ such that } \bigs|\g_\chi-\g\bigs| \leq \frac{\epsilon}{\log T}\bigg\}=0 .
        \]
\end{hypzerochi}

This roughly says that nontrivial zeros of a given Dirichlet $L$-function do not often cluster near nontrivial zeros of the zeta function. Analogous hypotheses have been used in a number of different settings (confer~\cite{BH 1, BH 2, C, CG, H, L}). 

Hypothesis $\mathscr{H}_{0, \chi}$ follows from an extension by Murty and Perelli~\cite{MP} of Montgomery's pair correlation conjecture to pairs of  functions in the Selberg class of $L$-functions. In the particular case of the Riemann zeta-function and a primitive Dirichlet $L$-function, they argue that if the GRH is true and if
$$
F(\a, T) = \frac{\pi}{T\log T} \sum_{-T\leq \g, \g_\chi\leq T} T^{i\a(\g-\g_\chi) }w(\g-\g_\chi),
$$
where $w(u)=\frac{4}{4+u^2}$, then one has
$$
F(\a, T) =
\begin{cases}
(1+o(1)) T^{-2\a} \log T+o(1) &\quad\hbox{if} \quad |\a|\leq 1,\\
o(1 ) &\quad\hbox{if} \quad |\a|\geq 1,
\end{cases}
$$
where the last estimate holds uniformly in any bounded $\a$ interval. As in Montgomery's work, the assertion for $|\a|\leq 1$ can be proved assuming GRH. From these estimates for $F(\alpha, T)$ they deduce that for $\d>0$,
\[
\sum_{-T\leq  \g, \g_\chi\leq T} \Bigg(\frac{\sin(\frac\d2(\g-\g_\chi)\log T)}{\frac\d2(\g-\g_\chi)\log T}\Bigg)^2
\sim \frac1\d \frac{T\log T}{\pi}.
\]
Now, $\big( {\sin x}/{x}\big)^2\geq \frac9{10}$ if $|x|\leq \frac{1}{2}$. Therefore, if $|\g-\g_\chi|\leq \e/\log T$ with 
$0<\e<1$, and we take $\d=1/\e$, then 
$$
\Bigg(\frac{\sin( ({1}/{2\e})(\g-\g_\chi)\log T)}{({1}/{2\e})(\g-\g_\chi)\log T}\Bigg)^2 \geq \frac{9}{10}.
$$
Thus we find that
\[
\frac{9}{10} \sum_{-T\leq  \g \leq T} \sum_{\substack{ -T\leq  \g_\chi \leq T\\  |\g-\g_\chi | \leq \e/\log T }} 1 
\leq \sum_{-T\leq  \g, \g_\chi\leq T} \Bigg(\frac{\sin(({1}/{2\e})(\g-\g_\chi)\log T)}{({1}/{2\e})(\g-\g_\chi)\log T}\Bigg)^2
\sim  \frac{\e T\log T}{\pi}.
\]
Letting $\e\to0$ gives Hypothesis $\mathscr{H}_{0, \chi}$.

Our first result is the following theorem.

\begin{thm}\label{distr of linear combination}
Let
\[
\mathcal{L}(\r)=a_1\log|L(\r, \chi_1)|+\dots+a_N\log|L(\r, \chi_N)|, \]
where $N$ is a fixed positive integer, $ a_1, \dots, a_N \in \mathbb R$,
and $\chi_1,\dots, \chi_N$ are distinct nonprincipal primitive Dirichlet characters. Assume the Generalized Riemann hypothesis and Hypothesis $\mathscr{D}.$ Further, suppose that for each $1\leq j \leq N$, Hypothesis $\mathscr{H}_{0,\chi_j}$ is true. Then for $A<B$ and $T$ sufficiently large, we have
\be
\begin{split}\label{eq in thm 1}
\frac{1}{N(T)}\#\bigg\{0< \g \leq T:  \frac{\mathcal{L}(\rho)}{\sqrt{\tfrac{1}{2}\big( {a_1}^2+\dots+{a_N}^2\big)
\log\log T}} \in [A, B]\bigg\} 
\sim \sdfrac{1}{\sqrt{2\pi}} \int_A^B e^{-\tfrac{x^2}{2}} \mathop{dx}  .
\end{split}
\ee
\end{thm} 
When a  quantity such as the one on the left-hand side of \eqref{eq in thm 1} tends to a normal distribution function as $T\to\infty$, we will say that 
 $$
  \frac{\mathcal{L}(\rho)}{\sqrt{\tfrac{1}{2}\big( {a_1}^2+\dots+{a_N}^2\big) \log\log T}}
 $$
has an \emph{approximately normal distribution} (in this case with mean $0$ and variance $1$).

On the assumption of the GRH and Hypothesis $\mathscr{D}$ alone (that is,  without assuming a zero-spacing hypothesis) one can prove a similar theorem for the linear combinations
\[
\mathcal A(\r) =a_1\arg L(\r, \chi_1)+\dots+a_n\arg L(\r, \chi_N)  .
\]
 (See  Theorem 1.2 in~\cite{C} for the case of $\arg\zeta(\r+z)$.) 

Before stating our next theorem, we remind the reader of some terminology from probability theory.
Let $\mathbf X =(X_1, X_2, \ldots, X_N)$ be a random vector in $\mathbb R^n$.  The mean or expectation of $X_i$  
is $\mathbb E(X_i) =\mu_i$, the variance is   $\mathrm{Var}(X_i)= \s_{i,i} = \mathbb E(X_i -\mu_i)^2$, and the covariance between  $X_i$ and $X_j$ is $\mathrm{Cov}(X_i, X_j)= \s_{i,j} = \mathbb E\big((X_i -\mu_i)(X_j -\mu_j)\big)$. We say that $\mathbf X$ has an N-variate normal distribution if its 
probability density function is given by
\be\notag
f_{\mathbf X}(\mathbf x) =\frac{1}{(2\pi)^{\frac N2} \sqrt{\mathrm{det} \mathbf C} }
e^{-\frac12 (\mathbf x-\boldsymbol \mu)^T \mathbf{C}^{-1} (\mathbf x-\boldsymbol \mu)},
\ee
where $\mathbf x=(x_1, x_2, \ldots, x_N)$,
$\boldsymbol \mu= (\mu_1, \mu_2,  \ldots, \mu_N)$, and $\mathbf C=(\s_{i, j})$ is the symmetric $N\times N$ covariance matrix.

Using Theorem~\ref{distr of linear combination}, we  will prove  
  
 \begin{thm}\label{n-variate} 
 Let $\chi_1, \chi_2, \ldots , \chi_N$ be distinct primitive Dirichlet characters, and assume the same hypotheses as in Theorem~\ref{distr of linear combination}. Then for $T$ sufficiently large and $0<\g\leq T$, the  vector 
\be\notag
\Bigg(
 \frac{ \log|L(\r, \chi_1)|}{\sqrt {\tfrac12 \log\log T}}, \ \dots \ ,\frac{\log|L(\r, \chi_N)|}{\sqrt {\tfrac12 \log\log T}} \Bigg)
\ee
has  approximately an $N$-variate normal distribution with mean vector $\mathbf 0_N$ and covariance matrix $\, \mathbf I_N$, the $N\times N$ identity matrix. In particular, the functions  $\log|L(\r, \chi_j)|/{\sqrt {\tfrac12 \log\log T}}, \\1\leq j \leq N$, are approximately mutually independent as $T\to \infty$.
 \end{thm}

\begin{cor*}\label{indep} 
    Let $ A_j < B_j $ for  $ 1\leq j\leq N$.
    Then under the same hypotheses as in Theorem~\ref{distr of linear combination} we have
            \begin{align*}
           \lim_{T\to\infty } \frac{1}{N(T)}
           \#\bigg\{0<\g& \leq T: \frac{\log|L(\r, \chi_j)|}{\sqrt {\tfrac12 \log\log T}} \in [A_j, B_j] \,\, \text{ for  } \,  1\leq j \leq N \bigg\}  
            \\
            =& \prod_{j=1}^N \lim_{T\to\infty }\frac{1}{N(T)}\#\bigg\{0<\g \leq T: 
            \frac{\log|L(\r, \chi_j)|}{\sqrt {\tfrac12 \log\log T}} \in [A_j, B_j] \bigg\}\\
            =&\Big(\frac1{2\pi}\Big)^{\frac N2}\prod_{j=1}^{N} \int_{A_j}^{B_j}e^{-\frac{x_j^2}{2}}  dx_j .
           \end{align*}
    \end{cor*}
The corollary follows immediately from Theorem~\ref{n-variate}.

Analogues of the results above (with appropriate modifications to our hypotheses) are true if one samples over the nontrivial zeros of any primitive Dirichlet $L$-function  not among $L(s, \chi_1), \ldots , L(s, \chi_N)$, rather than over the zeros of 
$\zeta(s)$.

As an application of the results above we consider the following question. Let 
$$
    F(s)=c_1L(s, \chi_1)+\dots+c_NL(s, \chi_N), 
$$
where  the $L(s,\chi_j)$  are distinct primitive Dirichlet $L$-functions and the coefficients $c_j$   are nonzero complex numbers. Given any complex number $a$, possibly $0$, how many solutions to 
$$
F(\r)=a
$$
are there when $\r=\frac12+i\g$ runs over the zeros of $\zeta(s)$  with $0<\g\leq T$\ ? 
We note that under similar assumptions as in the next theorem, but via a different method, the third author~\cite{Kirila}  proved that at least   $33\%$ of the nontrivial zeros of the Riemann zeta-function are not zeros of a linear combination of \emph{two} Dirichlet $L$-functions. Here we show that $100\%$ of the nontrivial zeros of the 
Riemann zeta-function are not zeros of an arbitrary linear combination of Dirichlet $L$-functions.

\begin{thm}\label{main thm 3}
Assume the Generalized Riemann hypothesis, Hypothesis $\mathscr D$, and Hypothesis $\mathscr{H}_{0, \chi_j}$ for  each of the primitive Dirichlet $L$-functions $L(s, \chi_j)$, $1\leq j \leq N$. 
    Let 
        \[
        F(s)=c_1L(s, \chi_1)+\dots+c_NL(s, \chi_N),
        \]
 where $ c_1, c_2, \dots, c_N$ are nonzero complex numbers. Then for any complex number  $a$ we have
  \be \notag
    \lim_{T\to\infty}\frac{1}{N(T)}\#\bigs\{0<\g \leq T: F(\rho)=a \bigs\} =0.
\ee
That is,  the proportion of  zeros $\r$ of $\zeta(s)$ for which $F(\r)=a$  has asymptotic density zero as $T\to \infty$. 
\end{thm}

Note that, as with our previous results,  Theorem~\ref{main thm 3} can also be modified so as to sample the linear combination $F$ over the zeros of an arbitrary primitive Dirichlet $L$-function that is different from any of those comprising the sum.


Throughout we take $T$ to be a sufficiently large positive real number. The letters $c$ and $D$ will stand for positive constants that may be different at each occurrence with $c$ standing for an absolute constant, whereas $D$ may depend on various parameters. The variables $p$ and $q,$ indexed or not, always denote prime numbers, and the variables $j, k, \ell, m$ and $n$ always denote nonnegative integers.
Finally, we define
\be\label{Psi}
\Psi(Y)=\sum_{p\leq Y} p^{-1}
\ee
so that, by Merten's estimate, $\Psi(Y) =\log\log Y +O(1)$ for $Y\geq 2$.


\section*{Acknowledgments}

This work began while Fatma \c{C}i\c{c}ek was visiting the Indian Institute of Technology Gandhinagar, whom she thanks for the hospitality. Her work on this project was later  supported by a postdoctoral fellowship of the Pacific Institute for the Mathematical Sciences at the University of Northern British Columbia.


\section{Lemmas for the proof of Theorem~\ref{distr of linear combination} }\label{sec: lemmas for Thm 1}

Most of the lemmas in this section are proved along the lines of similar ones for the zeta function in ~\cite{C}, so we will be content with sketching the proofs and referring the reader to arguments there.

To state our  lemmas we require some notation.
Let $L(s, \chi)$ be the Dirichlet $L$-function corresponding to a primitive Dirichlet character $\chi$ modulo $m$. 
For $T$ a large parameter and $X$ a parameter to be chosen later satisfying $4\le X \leq T^2$, we define 
 \be\notag
        \begin{split}
        w_X(n)=
        \begin{cases}
     \displaystyle   1&\quad \text{if} \quad 1\leq n \leq X,\\
       \displaystyle     \frac{\log{(X^2/n)}}{\log{X}} &\quad \text{if} \quad  X< n \leq X^2,\\
        0&\quad \text{if} \quad X^2 \leq n.
        \end{cases}
        \end{split}
\ee
We then put
$$
\Lambda_X(n)=\Lambda(n)w_X(n),
$$
where $\Lambda(n)$ is von Mangoldt's function.
We also write
\be\notag 
\CMcal{P}_{\chi}(\g)=\sum_{p\leq X^2}\frac{\chi(p)}{p^{1/2+i\g}} ,
\ee
and define
    \[
    \eta_{\chi}(\g)=\min_{\g_\chi} |\g-\g_\chi|,
    \]
where $\g$ is the ordinate of a zero of $\zeta(s)$  and $\g_\chi$ runs over the ordinates of nontrivial zeros $\r_\chi=\frac12+i\g_\chi$ of $L(s, \chi)$. Note that $ \eta_\chi(\g)$ never vanishes due to our assumption of Hypothesis $\mathscr{D}$. 

Our first lemma allows us to approximate $\log|L(\r, \chi)|$ by $\Re\CMcal{P}_{\chi}(\g)$.

\begin{lem}\label{log L approximate formula}
    Assume the GRH and Hypothesis $\mathscr{D}$. Let $4\leq X\leq T^2$, let $\chi$ be a primitive Dirichlet character modulo $m\geq 3$, $s=\s+it$, and $\s_1=\frac12+\frac{4}{\log X}$. Then
for $s\neq \r_\chi$ and $s$ not a trivial zero of $L(s, \chi)$, we have 
\be\notag
        \log|L(\r, \chi)|
        =\Re\CMcal{P}_{\chi}(\g)
        +O\Big(\sum_{j=1}^{4}r_{\chi, j}(X, \g)\Big)
        +O\Big(\frac{\log(m|\g|)}{\log X}\Big)+O(1),
\ee
where
        \be\notag
        \begin{split}
        r_{\chi, 1}(X, \g)&=\Big|\sum_{p\leq X^2}\frac{(1-w_X(p))\chi(p)}{p^{\frac12+i\g}}\Big|, \qquad
        r_{\chi, 2}(X, \g)=\Big|\sum_{p\leq X}\frac{w_X(p^2)\chi(p^2)}{p^{1+2i\g}}\Big|, \\
        &r_{\chi, 3}(X, \g)=\frac{1}{\log X} \int_{1/2}^{\infty} X^{\frac 1 2-\s}\Big|\sum_{p\leq X^2}\frac{\Lambda_X(p)\chi(p) \log(Xp)}{p^{\s+i\g}}\Big|\mathop{d\s},\quad \hbox{and}\\ \qquad
        &\qquad r_{\chi,4}(X, \g)=\Big(1+\log^+\frac{1}{\eta_{\chi}(\g)\log X}\Big)\frac{E_\chi(X, \g)}{\log X},
        \end{split}
        \ee
and
$$
        E_\chi(X, \g) =\Big|\sum_{n\leq X^2} \sdfrac{\Lambda_X(n)\chi(n)}{n^{\s_1+i\g}}\Big|+\log(m|\g|). 
$$
\end{lem}

\begin{proof}
First recall that  $L(s, \chi)$ has a trivial zero at $s=-2n-\mathfrak{a}$ for each positive integer $n$ where  
$\mathfrak{a}=0$ if $\chi(-1)=1$, and $\mathfrak{a}=1$ if $\chi(-1)=-1$. 
By a standard argument (see, for example, Lemma 15 in \cite{SelbergLFncs}), if 
     $s\neq \r_\chi$ or $-2n-\mathfrak{a}$ for any positive integer $n$, then we have
        \be\label{eq:L' by L}
        \begin{split}
        \frac{L'}{L}(s, \chi)
        =-&\sum_{n\leq X^2}\frac{\Lambda_X(n)\chi(n)}{n^{s}}  
       +\frac{1}{\log X}\sum_{\r_\chi} \frac{X^{\r_\chi-s}(1-X^{\r_\chi-s})}{(\r_\chi-s)^2} \\
        &+\frac{1}{\log X}\sum_{n=1}^\infty \frac{X^{-2n-\mathfrak{a}-s}(1-X^{-2n-\mathfrak{a}-s})}{(2n+\mathfrak{a}+s)^2}.
        \end{split}
        \ee
        
For the moment, suppose that $\s\geq \s_1$. Then 
\be\notag%
 \frac{L'}{L}(\s+it, \chi)
        =-\sum_{n\leq X^2}\frac{\Lambda_X(n)\chi(n)}{n^{\s+it}}  
       +\frac{2\omega X^{\frac12-\s}}{\log X}\sum_{\r_\chi} \frac{1}{(\s_1 -\frac12)^2+ (t-\g_\chi)^2}
       + O\Big(\frac{X^{-2-\s}}{t \log X}\Big),
\ee
where $|\omega|<1$.
This sum over nontrivial zeros can be estimated in terms of a sum over prime powers. Indeed,  from (5.3) and (5.8) of~\cite{SelbergLFncs} we have
        \be\label{sum over gamma chi}
        \frac{1}{\log X}\sum_{\g_{\chi}}\frac{1}{(\s_1-\frac12)^2+(t-\g_{\chi})^2}
        \ll  \sum_{\g_{\chi}}\frac{\s_1-\frac12}{(\s_1-\frac12)^2+(t-\g_{\chi})^2}
       \ll E_{\chi}(X, t),
        \ee
where 
        \be\notag
        E_{\chi}(X, t)=\Big|\sum_{n\leq X^2}\sdfrac{\Lambda_X(n)\chi(n)}{n^{\s_1+it}}\Big|+\log\big(m(1+|t|)\big).
        \ee
From this and \eqref{eq:L' by L}, we thus obtain for $ \s\geq \s_1$,
        \be \label{eq: L' L}
         \frac{L'}{L}(s,\chi)
        =- \sum_{n\leq X^2}\frac{\Lambda_X(n)\chi(n)}{n^s} 
        +O\bigs(X^{\frac12-\s}E_{\chi}(X, t)\bigs) .
        \ee
Also, we note the following estimate which holds for all $s$ with $\s\geq \frac12$ other than $\r_\chi$ (see~\cite[p. 83]{Davenport}).
        \be\label{L' by L over zero sum}
        \frac{L'}{L}(s, \chi)=\sum_{\r_\chi}\Big(\frac{1}{s-\r_\chi}-\frac{1}{\r_\chi}\Big)+O\bigs(\log\big(m(1+|t|)\big)\bigs).
        \ee 
Next, we  write
        \be\label{integral L' over L}
        \begin{split}
        \log |L(\r,\chi)|
        &=-\Re\int_{1/2}^{\infty} \frac{L'}{L}(\s+i\g,\chi)\mathop{d\s}  \\
        &=-\Re\int_{\s_1}^{\infty} \frac{L'}{L}(\s+i\g,\chi)\mathop{d\s} 
        -\Big(\s_1-\frac12\Big)\Re \frac{L'}{L}(\s_1+i\g,\chi)  \\ 
        &\qquad \quad 
        +\Re\int_{1/2}^{\s_1} \Big( \frac{L'}{L}(\s_1+i\g,\chi)- \frac{L'}{L}(\s+i\g,\chi) \Big)\mathop{d\s}
         \\ &=J_1+J_2+J_3.
        \end{split}
        \ee
For $J_1$, we integrate \eqref{eq: L' L} over $[\s_1,\infty)$ and obtain        \[
        J_1
        =\Re \sum_{n\leq X^2} \frac{\Lambda_X(n)\chi(n)}{n^{\s_1+i\g}\log n}
        +O\Big(\frac{E_{\chi}(X, \g)}{\log X}\Big).
        \]
Also, substituting $\s=\s_1$ in \eqref{eq: L' L}, we see that 
        \[
        J_2
        \ll \frac{1}{\log X}\Big|\frac{L'}{L}(\s_1+i\g, \chi)\Big|
        \ll \frac{E_{\chi}(X, \g)}{\log X}.
        \]
Using these estimates in \eqref{integral L' over L}, we find that
       \[
       \log|L(\r,\chi)|
       = \Re \sum_{n\leq X^2} \frac{\Lambda_X(n)\chi(n)}{n^{\s_1+i\g}\log n}
       +\Re{J_3}
        +O\Big(\frac{E_{\chi}(X, \g)}{\log X}\Big).
       \]
Now consider $\Re{J_3}$. By \eqref{L' by L over zero sum}, 
        \be\notag
        \begin{split}
        \Re&\Big\{\frac{L'}{L}(\s_1+i\g, \chi)-\frac{L'}{L}(\s+i\g, \chi)\Big\}\\
        &=\sum_{\g_{\chi}} \frac{\s_1-\frac12}{\big(\s_1-\frac12\big)^2+(\g-\g_{\chi})^2}
        -\sum_{\g_{\chi}} \frac{\s-\frac12}{\big(\s-\frac12\big)^2+(\g-\g_{\chi})^2} 
        +O(\log(m|\g|)).
        \end{split}
        \ee        
Then, by a similar argument  to that in the proof of Lemma 2.1 in ~\cite{C}, one finds that
        \be\notag
        \begin{split}
        \Re \,J_3
         \ll
        \Big(1  + \log^{+}\Big(\frac{1}{\eta_{\chi}(\g) \log X}\Big)\Big)  \sum_{\g_\chi}   \frac{(\s_1-\frac12)^2}{(\s_1-\frac12)^2 +(\g-\g_\chi)^2 }
         +O\Big(\frac{\log(m|\g|)}{\log X}\Big).
        \end{split}
        \ee
 Also, by \eqref{sum over gamma chi}, the sum over   $\g_\chi$ here is $O(E_\chi(X, \g)/\log X)$. Thus by \eqref{integral L' over L}, we may conclude that
\be\label{log L form}
        \log|L(\r,\chi)|
        = \Re \sum_{n\leq X^2} \sdfrac{\Lambda_X(n)\chi(n)}{n^{\s_1+i\g}\log n} 
        +O\Big(\Big(1+\log^+\Big(\frac{1}{\eta_{\chi}(\g)\log X}\Big)\Big)\frac{E_\chi(X, \g)}{\log X}\Big).
\ee
Using the argument in \cite[p. 35]{S1946Archiv}, one can easily see that the real part of the Dirichlet polynomial on the right is equal to
        \begin{align*} 
         \Re \CMcal{P}_{\chi}(\g)
        &+O\Big(\Big|\sum_{p\leq X^2}\sdfrac{(1-w_X(p))\chi(p)}{p^{\frac12+i\g}}\Big|\Big)
        +O\Big(\Big|\sum_{p\leq X}\sdfrac{w_X(p^2)\chi(p^2)}{p^{1+2i\g}}\Big|\Big)\\
        &+O\bigg( \frac{1}{\log X} \int_{1/2}^\infty X^{\frac 1 2-\s}\Big|\sum_{p\leq X^2}\sdfrac{\Lambda_X(p)\chi(p)\log{(Xp)}}{p^{\s+i\g}}\Big|\mathop{d\s}\bigg)+O(1). 
        \end{align*}
The result follows from this and \eqref{log L form}.
   \end{proof}    

Now, let $\chi_1, \ldots, \chi_N$ be distinct primitive Dirichlet characters with  conductors $m_1, \ldots, m_N$, respectively. For 
fixed nonzero real numbers $a_1, \ldots, a_N$, we set
\be\notag
     \mathcal{L}(\rho) = a_1\log|L(\r, \chi_1)|+a_2\log|L(\r, \chi_2)|+ \cdots+ a_N\log|L(\r, \chi_N)|
\ee
and 
\be\notag
    \CMcal{P}_{\mathcal{L}}(\g)
    =a_1 \CMcal{P}_{\chi_1}(\g)+\cdots +a_N \CMcal{P}_{\chi_N}(\g)
    = \sum_{p\leq X^2} \frac{a_1 \chi_1(p)+ \cdots+a_N \chi_N(p)}{p^{\frac12+i\g}}.
\ee


\begin{lem}\label{our lemma 3.16 linear combination}
 Assume RH. If $2\leq X\leq T^{\frac{1}{16k}}$ with $k$ a positive integer, then  
    \[
    \sum_{0 < \g \leq T}\big|\Re{\CMcal{P}_{\mathcal{L}}(\g)}\big|^k= O\bigs((Dk\Psi(X^2))^{\frac k2}N(T)\bigs), 
    \]
    where $\Psi(X^2)$ is as in \eqref{Psi}, and $D$ depends on the $a_1, \ldots, a_N$ and $N$.
\end{lem}

\begin{proof}
By an argument   along the lines of Proposition 4.1 in~\cite{C} and Stirling's formula, one can show that
    \[
    \sum_{0 < \g \leq T}\big|\Re{\CMcal{P}_{\chi_j}(\g)}\big|^k= O\bigs((ck\Psi(X^2))^{\frac k2}N(T)\bigs) 
    \]
for $j=1, 2, \ldots, N$ with $c$ an absolute constant. The result then follows easily upon using the inequality   $
         \Big(\frac 1n \sum_{i=1}^n x_i \Big)^k \leq \frac{1}{n} \sum_{i=1}^n x_i^k,
         $ 
     which is valid for $x_1, x_2, \dots, x_n>0$.
\end{proof}


The next result  is Lemma 3.3 of ~\cite{C}.

\begin{lem}\label{Dir poly mmt}
Assume RH. Let $k$ be a positive integer and suppose that $1<Y\leq (T/\log T)^{\frac{1}{3k}}$.
 For any complex-valued sequence $\{a_p\}_p$ indexed by the primes, we have
\be\notag
\sum_{0<\g\leq T} \bigg| \sum_{p\leq Y} \frac{a_p}{p^{\frac12+i\g}}\bigg| 
\ll k! N(T)   \Big(\sum_{p\leq Y} \frac{|a_p|^2}{p} \Big)^{k}.
\ee
\end{lem}

Next we introduce   random polynomials that in certain respects model the behavior of the Dirichlet polynomials $\CMcal{P}_{\mathcal{\chi}}(\g)$ and $\CMcal{P}_{\mathcal{L}}(\g)$. Let $\{\theta_p \}$ be a sequence indexed by the primes of identically distributed independent random variables, each of which is uniformly distributed over $[0, 1]$. The random polynomial modeling $\CMcal{P}_{\mathcal{\chi}}(\g)$ is
    \[
    \CMcal{P}_{\chi}(\underline{\theta})= \sum_{p\leq X^2}\frac{\chi(p)e^{2\pi i \theta_{p}}}{\sqrt p}.
    \]
 Here $\Re\CMcal{P}_{\chi}(\underline{\theta})$ has mean $0$ and variance 
  \[
   \int_0^1 \big(\Re\CMcal{P}_{\chi_i}(\underline{\theta})\bigs)^2 
    \mathop{d\underline{\theta}}
   = \tfrac12 \sum_{p\leq X^2}|\chi_i(p)|^2p^{-1}
    = \tfrac{1}{2}\Psi(X^2)+O_{m_i}(1).
  \]
Now suppose that an  integer $n>1$ has the unique factorization 
$\displaystyle n=p_1^{\mu_1}\dots p_r^{\mu_r}$ and write
    \[
    \theta_n=\mu_1\theta_{p_1}+\dots+\mu_r\theta_{p_r}.
    \]    
It is easy to see that
$\theta_{mn}=\theta_m+\theta_n $ for any  $\mkern9mu m, n \in \mathbb{Z}^+ $ and that  $\theta_m=\theta_n$ if and only if $ m=n$. Thus
    \be\label{orthogonality}
    \int_0^1 e^{2\pi i\theta_m} e^{-2\pi i\theta_n} \mathop{d\underline{\theta}} = 
    \begin{cases}
    1 &\text{if} \quad m=n, \\
    0 &\text{if} \quad m\neq n,
    \end{cases}
    \ee
where here and from now on, $\int_0^1 (\dots) \mathop{d\underline{\theta}}$ denotes a multidimensional integral.  
The random polynomial we use to model $\CMcal{P}_{\mathcal{L}}(\g)$ is
\be\label{P_L theta}\begin{split}
\CMcal{P}_{\mathcal{L}}(\underline{\theta}) 
= a_1\CMcal{P}_{\chi_1}(\underline{\theta})
+\cdots +a_N\CMcal{P}_{\chi_N}(\underline{\theta}) 
= \sum_{p\leq X^2}\frac{(a_1\chi_1(p)+
\cdots +a_N \chi_N(p))e^{2\pi i \theta_{p}}}{\sqrt p}.
\end{split}
\ee
Clearly the mean of $\Re \CMcal{P}_{\mathcal{L}}(\theta)$ is
    \[
    \begin{split}
    \int_0^1 \Re\CMcal{P}_{\mathcal{L}}(\underline{\theta}) \mathop{d\underline{\theta}} 
    =a_1  \int_0^1 \Big(\Re \sum_{p\leq X^2} \frac{\chi_1(p)e^{2\pi i\theta_p}}{\sqrt p}\Big) \mathop{d\underline{\theta}}
    + \cdots +a_N \int_0^1 \Big(\Re \sum_{p\leq X^2} \frac{\chi_N(p)e^{2\pi i\theta_p}}{\sqrt p}\Big) \mathop{d\underline{\theta}}
    =0.
    \end{split}
    \]
Its variance is
    \be\label{compute variance linear combination}
    \begin{split}
    \int_0^1\big(\Re\CMcal{P}_{\mathcal{L}}(\underline{\theta})\big)^2 \mathop{d\underline{\theta}} 
    =& \sum_{i=1}^N {a_i}^2 
    \int_0^1 \big(\Re\CMcal{P}_{\chi_i}(\underline{\theta})\bigs)^2 
    \mathop{d\underline{\theta}}\\
    &+ \sum_{1\leq i\neq j \leq N}a_i a_j \int_0^1 \Re\CMcal{P}_{\chi_i}(\underline{\theta}) 
    \Re\CMcal{P}_{\chi_j}(\underline{\theta})  \mathop{d\underline{\theta}}. 
    \end{split}
    \ee
As we saw above,
$
   \int_0^1 (\Re\CMcal{P}_{\chi_i}(\underline{\theta}))^2 
    \mathop{d\underline{\theta}}
    = \tfrac{1}{2}{\Psi(X^2)}+O_{m_i}(1).$
    Thus, the first term on the right-hand side of \eqref{compute variance linear combination} equals
$$
\tfrac12 \Psi(X^2)  \sum_{i=1}^N a_i^2 +O(1),
$$
where the $O$-term constant depends on the $a_i$ and $m_i$.
For $i\neq j$, the other integral on the right-hand side of \eqref{compute variance linear combination} equals
    \begin{align*}
    \int_0^1\bigs(\Re\CMcal{P}_{\chi_i}(\underline{\theta}) \bigs)
    \bigs( \Re\CMcal{P}_{\chi_j}(\underline{\theta})\bigs)    \mathop{d\underline{\theta}}  
    &= \frac14 \int_0^1 \bigs(\CMcal{P}_{\chi_i}(\underline{\theta}) +
    \overline{\CMcal{P}_{\chi_i}(\underline{\theta})}\bigs)
     \bigs(\CMcal{P}_{\chi_j}(\underline{\theta}) +
    \overline{\CMcal{P}_{\chi_j}(\underline{\theta})}\bigs) 
    \mathop{d\underline{\theta}} \\
    &=  \frac12 \int_0^1 \Re\bigs(\CMcal{P}_{\chi_i}(\underline{\theta}) \CMcal{P}_{\chi_j}(\underline{\theta})\bigs)
     \mathop{d\underline{\theta}}
    +  \frac12 \int_0^1 \Re\bigs(\CMcal{P}_{\chi_i}(\underline{\theta})\overline{\CMcal{P}_{\chi_j}(\underline{\theta})} \, \bigs) \mathop{d\underline{\theta}}. 
    \end{align*}
By \eqref{orthogonality} we have
    \[
    \Re\int_0^1\CMcal{P}_{\chi_i}(\underline{\theta}) \CMcal{P}_{\chi_j}(\underline{\theta})\mathop{d\underline{\theta}}
      =0
    \]
and
    \[
    \Re \int_0^1\CMcal{P}_{\chi_i}(\underline{\theta})\overline{\CMcal{P}_{\chi_j}(\underline{\theta})}  \mathop{d\underline{\theta}}
    = \Re\sum_{p\leq X^2} \frac{\chi_i(p) \overline{\chi}_j(p)}{p}.
    \]
Here $\chi_i\overline{\chi_j}$ is a nonprincipal Dirichlet character modulo $M_{i,j}=\operatorname{lcm}[m_i, m_j]$, so we have
     \[
     \sum_{p\leq X^2} \frac{\chi_i(p)\overline{\chi_j}(p)}{p}
     \ll_{M_{i, j}}1. 
     \]
Thus, the second term in \eqref{compute variance linear combination}  is 
$$
\ll \sum_{1\leq i\neq j \leq N}|a_i a_j| \ll N \sum_{i=1}^N a_i^2,  
$$
where the implied constant depends on $m_1, \ldots,m_N$. 
 By \eqref{compute variance linear combination}, we now see that the variance of 
 $\Re\CMcal{P}_{\mathcal{L}}(\underline{\theta})$ is
    \begin{equation}\notag
      \int_0^1(\Re\CMcal{P}_{\mathcal{L}}(\underline{\theta}))^2 \mathop{d\underline{\theta}}
    =\tfrac{1}{2}\Psi(X^2)(a_1^2+{a_2}^2+\cdots+ a_N^2 )+O(1),
    \end{equation}
    with the $O$-term constant depending on the $a_i, m_i$, and $N$.
Notice that for  $X$ a small power of $T$, this is $\sim\frac{{a_1}^2+ \cdots+{a_N}^2}{2} \log\log T$, which is the 
variance  of the approximate normal distribution in the statement of Theorem \ref{distr of linear combination}. 

As an analogue of Lemma \ref{our lemma 3.16 linear combination} for $\Re\CMcal{P}_{\mathcal{L}}(\underline{\theta})$ we have

\begin{lem}\label{Tsang lemma 3.4 linear combination}
    Suppose that $k$ is a nonnegative integer. For $\displaystyle 2\leq X \leq T^{\frac{1}{2k}}$ we have
    \be\notag
    \int_0^1|\Re{\CMcal{P}_{\mathcal{L}}(\underline{\theta})}|^k\mathop{d{\underline{\theta}}}
    = O ((Dk\Psi(X^2))^{\frac k2} ).
    \ee
    Here $D$ depends at most on the $a_i, m_i,$ and $N$.
\end{lem}

\begin{proof}
As in (3.10) of \cite{Tsang}, one can use \eqref{orthogonality} and easily prove that 
    \be\notag
    \int_0^1|\Re{\CMcal{P}_{\chi_j}(\underline{\theta})}|^k\mathop{d{\underline{\theta}}}
    = O((ck\Psi(X^2))^{\frac{k}{2}}) \quad \text{for} \quad j=1, 2, \ldots, N. 
    \ee
The claim of the lemma is then straightforward.    
\end{proof}

Next we express the $k^{th}$ moment of $\Re\CMcal{P}_{\mathcal{L}}(\g)$ in terms of integrals of the $k^{th}$ moment of $\Re{\CMcal{P}_{\mathcal{L}}(\underline{\theta})}$. 
  
\begin{lem}\label{lemma 3.4 linear combination}
Assume RH. 
    Let $k$ be a positive integer with $k \ll \sqrt{\log T},$ and suppose that  $\displaystyle X\leq T^{\frac{1}{16k}}$. Then
     \be\label{eq:lemma 3.4 linear combination}
    \begin{split}
    \sum_{0 < \g  \leq T} \bigs(\Re\CMcal{P}_{\mathcal{L}}(\g)\bigs)^k 
    = &N(T)\int_0^1  (\Re\CMcal{P}_{\mathcal{L}}(\underline{\theta}))^k\mathop{d{\underline{\theta}}}  \\
    &-\frac{T}{\pi}\sum_{q\leq X^2} \sum_{1\leq \ell\leq k} \frac{\log q}{q^{\frac{\ell}{2}}}
    \int_0^1 (\Re\CMcal{P}_{\mathcal{L}}(\underline{\theta}))^k 
    \Re(e^{2\pi i\ell\theta_q}) \mathop{d{\underline{\theta}}} \\
    &\quad +O ( (Dk)^k\sqrt T\log^3{T}).
    \end{split}
    \ee
    Here the constant $D$ in the $O$-term depends on $a_1, \ldots, a_N$ and $N.$
\end{lem}
{\emph{Remark.} } Notice that this holds without the $O$-term when $k=0$ if we  interpret the  sum over $\ell$ on the right to be empty in that case.
\begin{proof}
By the binomial theorem,
        \[
        \sum_{0 < \g \leq T} (\Re \CMcal{P}_{\mathcal{L}}(\g))^k
        =\frac{1}{2^k}\sum_{j=0}^k \binom{k}{j}
        \sum_{0 < \g \leq T}  \CMcal{P}_{\mathcal{L}}(\g)^j 
        \overline{\CMcal{P}_{\mathcal{L}}(\g)}^{k-j}.
        \]
For $0\leq j\leq k$, 
        \be\label{j term linear combination}
        \begin{split}
        \sum_{0 < \g \leq T}\CMcal{P}_{\mathcal L}(\g)^j 
         \overline{\CMcal{P}_{\mathcal L}(\g)}^{k-j}  
        =\sum_{0 < \g \leq T}\Big(\sum_{p\leq X^2} & \sdfrac{a_1\chi_1(p)+ \cdots +a_N\chi_N(p)}{p^{\frac12+i\g}}\Big)^j \\& \cdot \Big(\overline{\sum_{p\leq X^2}\sdfrac{a_1\chi_1(p)+\cdots+ a_N\chi_N(p)}{p^{\frac12+i\g}}}\Big)^{k-j}.
       \end{split}
        \ee
We will write
     \[
     \CMcal{P}_{\mathcal L}(\g)^j
     =\sum_{\substack{n=p_1\dots p_j,\\ p_i \leq X^2}}  \frac{b_j(n)\kappa(n)}{n^{\frac12+i\g}} 
     \quad \text{and}
     \quad
      \CMcal{P}_{\mathcal L}(\g)^{k-j}
     =\sum_{\substack{m=q_1\dots q_{k-j},\\ q_i \leq X^2}}  \frac{b_{k-j}(m)\kappa(m)}{m^{\frac12+i\g}} .
     \]   
     Here $p_1,\dots, p_j, q_1,\dots, q_{k-j}$ denote primes that are not necessarily distinct,   the coefficient $b_r(p_{1}\dots p_{r})$ denotes the number of permutations of the primes $p_{1},\dots, p_{r}$, and $b_r(n)=0$ if $n$ is not the product of $r$ primes. We have also 
     written    \[
    \kappa(p_{1}\dots p_{r}) 
    = \bigs(a_1\chi_1(p_{1})+\cdots + a_N\chi_N(p_{1})\bigs) 
    \dots \bigs(a_1\chi_1(p_{r})+\cdots + a_N\chi_N(p_{r})\bigs). 
    \] 
Using these notations and Corollary 3.2 in~\cite{C}, we find that \eqref{j term linear combination} equals
         \[
        \begin{split}
        N(T)&\sum_{n} \frac{b_j(n)\kappa(n)}{\sqrt n} \frac{b_{k-j}(n)\overline{\kappa(n)}}{\sqrt n}  \\
        -&\frac{T}{2\pi} \sum_{n}\sum_{m} 
        \frac{b_j(n)\kappa(n)}{\sqrt n}
        \frac{b_{k-j}(m)\overline{\kappa(m)}}{\sqrt m}
        \bigg(\frac{\Lambda(m/n)}{\sqrt{m/n}}+\frac{\Lambda(n/m)}{\sqrt{n/m}} \bigg) \\
         &+O \Big(X^{2k}\log T \log\log T\sum_n\sdfrac{b_j(n)|\kappa(n)|}{n}
         \sum_{n<m}b_{k-j}(m) |\kappa(m)| \Big) \\
        &+O \Big(X^{2k}\log T \log\log T
        +\sum_m\sdfrac{b_{k-j}(m) |\kappa(m)|}{m}\sum_{m<n} b_j(n) |\kappa(n)| \Big) \\
        &+O \Big(X^{2k}\log^2{T}\Big(\sum_n\sdfrac{b_j(n)^2|\kappa(n)|^2}{n}
        +\sum_m\sdfrac{b_{k-j}(m)^2 |\kappa(m)|^2 }{m}\Big) \Big).
        \end{split}
        \]
In this we again have $n=p_1\dots p_j$ and $m=q_1\dots q_{k-j}$. The rest of the proof is similar to that of Lemma 5.1 in~\cite{C}. The two main terms in the above will match the main terms on the right-hand side of \eqref{eq:lemma 3.4 linear combination}, and one  uses the bound $|\kappa(p_{i_1}\dots p_{i_r})| \leq (|a_1|+\cdots+|a_N|)^r$ to estimate the error terms.
\end{proof}

  Recalling that   $ \Psi (Y)=\sum_{p\leq Y}p^{-1}$,
  we will use the following parameters in  subsequent arguments:
  \be\label{K X}
    K=2[\Psi(T)^6]  \quad
     \text{and} \quad X \leq T^{\frac{1}{16\Psi(T)^6}}.
    \ee
    Here $[\cdot]$ denotes the greatest integer function.
\begin{lem}\label{fourier linear combination}
Assume RH.
Let $K$ and $X$ be as in \eqref{K X}. Then for $\omega\in [0, (\log\log T)^2]$ we have
        \be\label{eq:lemma fourier}
        \begin{split}
        \sum_{0 < \g \leq T}\exp( i \omega&\Re{\CMcal{P}_{\mathcal L}(\g)}) 
        = N(T)\int_0^1
         \exp( i \omega \Re{\CMcal{P}_{\mathcal L}(\underline{\theta})})\mathop{d\underline{\theta}} \\
        &-\frac{T}{\pi}\sum_{q\leq X^2} \sum_{1\leq\ell\leq K} \frac{\log q}{q^{\frac{\ell}{2}}}
        \int_0^1\exp(i\omega\Re{\CMcal{P}_{\mathcal L}(\underline{\theta})})
        \Re( e^{2\pi i\ell\theta_q}) \mathop{d\underline{\theta}} 
        +O\Big( \frac{N(T)\omega}{2^K}\Big).
        \end{split}
        \ee
        The implied constant in the $O$-term depends on $a_1, \ldots, a_N$ and $N$.
\end{lem}

\begin{proof}
The proof is very similar to that of Lemma 5.4 in~\cite{C}. 
The basic idea is to use Lemma \ref{lemma 3.4 linear combination} together with the estimate 
    \be\label{eq:exponential}
    e^{ix}=\sum_{0\leq k<K}\frac{(ix)^k}{k!}+O\Big(\frac{|x|^K}{K!}\Big) \qquad\qquad
    (x\in\mathbb{R}).
    \ee     
By \eqref{eq:exponential}, 
        \be\label{eq:exponential expansion linear combination}
        \begin{split}
        \sum_{0 < \g \leq T} \exp(i \omega\Re{\CMcal{P}_{\mathcal L}(\g)}) 
        =\sum_{0\leq k< K} & \frac{( i \omega)^k}{k!}
        \sum_{0 < \g \leq T} (\Re{\CMcal{P}_{\mathcal L}(\g)} )^k  \\
        &+O\Big(\frac{\omega^K}{K!}\sum_{0 < \g \leq T}
        |\Re{\CMcal{P}_{\mathcal L}(\g)}|^K\Big).
        \end{split}
        \ee
 By Lemma \ref{our lemma 3.16 linear combination} and Stirling's formula, the $O$-term is    
    \be\notag
    \begin{split}
        &\ll N(T)\frac{\omega^K}{K!}(DK\Psi(X^2))^{\frac K2} 
        \ll N(T) \omega \, \frac{D^K\omega^{K-1} }{K^K}(DK\Psi(X^2))^{\frac K2} \\
       & \ll N(T)\omega\Big(\frac{D ( \log\log T)^2 \sqrt{\Psi(T)}}{\sqrt K}\Big)^K
        \ll N(T)\frac{\omega}{2^{K}},
\end{split}
\ee
where the implied constant depends, at most, on $a_1, \ldots, a_N$ and $N$.
Now we apply Lemma \ref{lemma 3.4 linear combination} to the main term on the right-hand side of \eqref{eq:exponential expansion linear combination} to see that 
        \be\label{eq:fourier linear combination} 
        \begin{split}
        \sum_{0\leq k< K}\frac{( i \omega)^k}{k!}\sum_{0 < \g \leq T} 
        &(\Re{\CMcal{P}_{\mathcal L}(\g)})^k 
        =N(T)\sum_{0\leq k< K}\frac{( i \omega)^k}{k!}\int_0^1
         ( \Re{\CMcal{P}_{\mathcal L}(\underline{\theta})})^k \mathop{d\underline{\theta}}  \\
        &-\frac{T}{\pi}
         \sum_{1\leq k< K}\frac{(i \omega)^k}{k!} 
         \sum_{q\leq X^2}  \sum_{1\leq\ell\leq k}
        \frac{\log q}{q^{\ell/2}}\int_0^1
         ( \Re{\CMcal{P}_{\mathcal L}(\underline{\theta})} )^k 
        \Re( e^{2\pi i\ell\theta_q})\mathop{d{\underline{\theta}}}   \\
         &+O\Big(\sqrt T\log^3{T}\sum_{1\leq k< K}\frac{ \omega^k}{k!}(Dk)^k\Big). \\
        &= S_1+ S_2+S_3 . 
        \end{split}
        \ee
Notice that there is no $k=0$ term  in either $S_2$ or $S_3$  in light of  the remark after Lemma \ref{lemma 3.4 linear combination}.

The estimation of $S_1$ is straightforward. By  \eqref{K X} and \eqref{eq:exponential} we find that
        \be\label{S1 fourier linear combination}
        S_1
        =N(T)\int_0^1 \exp( i \omega \Re{\CMcal{P}_{\mathcal L}(\underline{\theta})})
         \mathop{d\underline{\theta}} 
        +O\Big(N(T) \frac{\omega}{2^{K}}\Big).
        \ee
The estimation of $S_3$ is also not difficult. 
Our choice of parameters in \eqref{K X} leads to
        \be\label{S3 fourier linear combination}
        S_3 
        \ll \sqrt{T}\log^3{T} \sum_{1\leq k< K}\frac{\omega^k}{k!}(Dk)^k
        \ll  \omega \, (\log\log T)^{2K}\sqrt{T}\log^3{T} \,  e^{DK}
        \ll    \frac{N(T)\omega}{2^K},
        \ee      
which is crude but sufficient for our purpose.
        
The estimation of $S_2$ 
is more involved. First, we extend the inner sum in $S_2$ from 
   $1\leq \ell \leq k$ to $1\leq \ell \leq K.$ To justify this, notice that by the binomial theorem,  
        \[
        \int_0^1\bigs( \Re{\CMcal{P}_{\mathcal L}(\underline{\theta})}\bigs)^k e^{2\pi i\ell\theta_q}\mathop{d{\underline{\theta}}} 
        = \frac{1}{2^k} \sum_{j=0}^k \binom{k}{j} \\
       \int_0^1 \CMcal{P}_{\mathcal L}(\underline{\theta})^{j} \overline{\CMcal{P}_{\mathcal L}(\underline{\theta})}^{k-j}
        e^{2\pi i\ell\theta_q} \mathop{d{\underline{\theta}}}.
        \]
Multiplying out the integrand and using \eqref{orthogonality}, we see that the $j$th term  is $0$ unless $k-j=j+\ell$. Since  $0\leq j\leq k,$ this implies that $1\leq \ell \leq k.$ 
The same is true for the complex conjugate of this integral. Thus, the integral in $S_2$ is $0$ for $\ell> k$ and we can write
    \be\label{S2 fourier linear combination}
    S_2
    =  -\frac{T}{\pi}\sum_{q\leq X^2} \sum_{1\leq\ell\leq K}
     \frac{\log q}{q^{\frac{\ell}{2}}} I_\ell (q),
    \ee
where
    \[
    I_\ell (q)= \int_0^1
         \bigg( \sum_{0\leq k< K} 
         \frac{\bigs( i \omega \Re{\CMcal{P}_{\mathcal L}(\underline{\theta})}\bigs)^k}{k!}  \bigg)
         \Re (e^{2\pi i\ell\theta_q})\mathop{d\underline{\theta}}.
    \]    
The sum over $k$ in $I_\ell(q)$ is approximately $\exp\bigs( i\omega \Re{\CMcal{P}_{\mathcal L}(\underline{\theta})}\bigs)$. In fact, if $\ell>1$, then 
    \be\label{I_ell linear combination}
    I_\ell(q)= \int_0^1 \exp\bigs( i\omega \Re{\CMcal{P}_{\mathcal L}(\underline{\theta})}\bigs)
     \Re (e^{2\pi i\ell\theta_q}) \mathop{d\underline{\theta}}
    +O\bigg(\frac{\omega}{2^K}\bigg)
    \ee
by \eqref{eq:exponential}, Lemma \ref{Tsang lemma 3.4 linear combination}, and our choice of parameters in \eqref{K X}. If $\ell=1$, the analysis  is more intricate. We begin by writing   
       \be\label{I1+T1 linear combination}
       \begin{split}     
           \int_0^1 \exp\bigs( i\omega \Re{\CMcal{P}_{\mathcal L}(\underline{\theta})}\bigs)
           \Re(e^{2\pi i  \theta_q})\mathop{d\underline{\theta}} 
            =  I_1(q) +  T_1(q),
        \end{split}
         \ee   
where
    \[
    T_1(q)
     =\sum_{k\geq  K}\frac{( i \omega)^k}{k!} 
           \int_0^1( \Re{\CMcal{P}_{\mathcal L}(\underline{\theta})})^k 
           \Re(e^{2\pi i\theta_q})\mathop{d{\underline{\theta}}}.
    \]         
We wish to show that $T_1(q)$ is small. First consider
    \[
    {T_1}^{+}(q)
       = \sum_{k\geq K}\frac{( i \omega)^k}{k!} 
           \int_0^1\bigs( \Re{\CMcal{P}_{\mathcal L}(\underline{\theta})}\bigs)^k e^{2\pi i\theta_q}\mathop{d{\underline{\theta}}}.
    \]  
The integral here is
        \begin{align*}
         \int_0^1\bigs( \Re{\CMcal{P}_{\mathcal L}(\underline{\theta})}\bigs)^k e^{2\pi i\theta_q}\mathop{d{\underline{\theta}}}= 
         \frac{1}{2^k} \sum_{j=0}^k \binom{k}{j}   
            \int_0^1 \bigg(\sum_{p\leq X^2}\frac{(a_1\chi_1(p)+\ldots+a_N\chi_N(p))e^{2\pi i\theta_p}}{\sqrt p}\bigg)^{j}& \\
        \cdot\bigg(\overline{\sum_{p\leq X^2}\frac{(a_1\chi_1(p)+\ldots+a_N\chi_N(p))e^{2\pi i\theta_p}}{\sqrt p}}\bigg)^{k-j}e^{2\pi i\theta_q}
        \mathop{d{\underline{\theta}}}&.
        \end{align*}
By \eqref{orthogonality}, the integral  on the right-hand side equals zero unless $j+1=k-j,$ that is, unless $j=(k-1)/{2}.$ In that case, we have $q_1\dots q_{(k+1)/2}=p_1\dots p_{(k-1)/2}\,q\,$ for some primes $p_1,\dots,p_{(k-1)/2}$ and  $q_1,\dots, q_{(k+1)/2}.$
Thus
\be\notag
\begin{split}
   \int_0^1 &\bigs( \Re{\CMcal{P}_{\mathcal L}(\underline{\theta})}\bigs)^k 
   e^{2\pi i\theta_q}\mathop{d{\underline{\theta}}} \\
   &=  \frac{1}{2^k\sqrt q}
    \binom{k}{ \frac{k-1}{2}} 
   \sum_{p_1,\dots,p_{(k-1)/2}  \leq X^2} 
    \sdfrac{b_{\frac{k+1}{2}}( p_1\dots p_{(k-1)/2} q) \varkappa_{\frac{k-1}{2}}(p_1\dots p_{(k-1)/2})
   \overline{\varkappa_{\frac{k+1}{2}}( p_1\dots p_{(k-1)/2}q)}}{p_1\dots p_{(k-1)/2}},
\end{split}
\ee
   where
   \[
   \varkappa_r(p_{1}\dots p_{r})
   = \prod_{i=1}^{r}(a_1\chi_1(p_{i})+\cdots +a_N\chi_N(p_{i})),
   \]
   and $b_r(p_{1}\dots p_{r})$ denotes the number of permutations of primes $p_{1},\dots, p_{r}$. Clearly 
   $$
     b_{(k+1)/2}( p_1\dots p_{(k-1)/2}\ q)
     \leq \big(\tfrac{k+1}{2}\big)!, 
$$
 and 
$$
\prod_{i=1}^r  |(a_1\chi_1(p_i)+\cdots +a_N\chi_N(p_i)|     \leq (|a_1|+\cdots+|a_N|)^r.
$$
    From these bounds we obtain
    \begin{align*}
     \sum_{p_1,\dots,p_{(k-1)/2}  \leq X^2}& \sdfrac{b_{(k+1)/2}( p_1\dots p_{(k-1)/2}\ q) \varkappa_{(k-1)/2}(p_1\dots p_{(k-1)/2})
   \overline{\varkappa_{(k+1)/2}( p_1\dots p_{(k-1)/2}q)}}{p_1\dots p_{(k-1)/2}} \\
    & \leq (|a_1|+\cdots +|a_N|)^{\tfrac{k-1}{2}} \big(\tfrac{k+1}{2}\big)! \sum_{p_1,\dots,p_{(k-1)/2}  \leq X^2}
    \frac{|\varkappa_{(k+1)/2}(p_1\dots p_{(k-1)/2}\ q)|}{p_1\dots p_{(k-1)/2}} \\
    &\leq (|a_1|+\cdots+|a_N|)^{k}  \big(\tfrac{k+1}{2}\big)! \Big(\sum_{p\leq X^2}\frac{1}{p}\Big)^{\frac{k-1}{2}}
    \leq k  \big(\tfrac{k-1}{2}\big)! (D\Psi(X^2))^{\frac{k-1}{2}}.
    \end{align*}
    Hence 
     \[
      \int_0^1\bigs( \Re {\CMcal{P}_{\mathcal L}(\underline{\theta})}\bigs)^k e^{2\pi i\theta_q}\mathop{d{\underline{\theta}}}
      \ll \frac{k}{\sqrt q}\binom{k}{ \frac{k-1}{2}} \big(\tfrac{k-1}{2}\big)! (D\Psi(X^2))^{\frac{k-1}{2}}
      \ll \frac{k}{\sqrt q} (Dk\Psi(X^2))^{\frac{k-1}{2}},
     \]
   and therefore
        \begin{align*}
        {T_1}^{+}(q)
        \ll & \frac1{\sqrt q} \sum_{k\geq K} k \frac{ \omega^k}{k!}(Dk\Psi(X^2))^{\frac{k-1}{2}} 
        \ll  \frac{ \omega }{\sqrt q} \sum_{k\geq K} \frac{( D\, (\log\log T)^2 \sqrt{k\Psi(X^2)})^{k-1}}{(k-1)!}  \\
         \ll &   \frac{ \omega }{\sqrt q} \sum_{k\geq K}  (  D (\log\log T)^2  \sqrt{ {\Psi(X^2)}/{k} })^{k-1}.
        \end{align*}
By our choice of parameters in \eqref{K X} we see that the terms of the sum are $\ll 2^{-k}$ when $T$ is sufficiently large. Thus
  \[
    {T_1}^{+}(q)
             \ll  \frac{\omega}{2^K \sqrt q}.
    \]  
Similarly, we have
\[
  {T_1}^{-}(q)= \sum_{k\geq K}\frac{( i \omega)^k}{k!} 
       \int_0^1\bigs( \Re {\CMcal{P}_{\mathcal L}(\underline{\theta})}\bigs)^k e^{-2\pi i\theta_q}\mathop{d{\underline{\theta}}}
       \ll  \frac{\omega}{2^K\sqrt q}.
\]
It follows that $T_1(q) \ll \omega/(2^K\sqrt q)$.  Hence, from \eqref{I1+T1 linear combination} we obtain
    \[
     I_1(q) 
     =\int_0^1   \exp\bigs( i\omega \Re {\CMcal{P}_{\mathcal L}(\underline{\theta})}\bigs)
     \Re(e^{2\pi i  \theta_q})\mathop{d\underline{\theta}} 
    +O\Big(\frac{\omega}{2^K \sqrt q}\Big).
     \]
Combining this with \eqref{I_ell linear combination} in  \eqref{S2 fourier linear combination}, we find that
    \[
     S_2  
     =-\frac{T}{\pi} \sum_{q\leq X^2}   \sum_{ 1\leq\ell\leq K}
     \frac{\log q}{q^{\frac{\ell}{2}}} \; 
        \int_0^1   \exp{  \bigs(i\omega \Re{\CMcal{P}_{\mathcal L}(\underline{\theta})}\bigs)}
        \Re(e^{2\pi i \ell \theta_q}) \mathop{d\underline{\theta}}  
        +O\Big(\frac{\omega T}{2^K} \sum_{q\leq X^2} \frac{\log q}{q}\Big).
      \]
The sum in the error term is $O( \log T),$ so 
     \[
      S_2 
       =-\frac{T}{\pi}  \sum_{q\leq X^2}   \sum_{ 1\leq\ell\leq K}
       \frac{\log q}{q^{\frac{\ell}{2}}} \; 
        \int_0^1 \exp\bigs( i\omega \Re {\CMcal{P}_{\mathcal L}(\underline{\theta})}\bigs)
        \Re(e^{2\pi i \ell \theta_q})\mathop{d\underline{\theta}} 
        +O\Big(\frac{\omega N(T)}{2^K}\Big) .
      \]
Inserting  this estimate, the results of \eqref{S1 fourier linear combination}, and \eqref{S3 fourier linear combination} into \eqref{eq:fourier linear combination}, we obtain the assertion of the lemma.
 \end{proof}


\begin{cor}\label{cor: fourier}
Let $K$ and $X$ be as in \eqref{K X}. Then for $\omega\in [0, \sqrt{2/N}(\log\log T)^{\frac12-\e}]$ with $0<\e<\frac12$, we have
\[
        \sum_{0 < \g \leq T}\exp\bigg(  i \omega \sdfrac{\Re{\CMcal{P}_{\mathcal L}(\g)}}{\sqrt{\tfrac{1}{2}\big( {a_1}^2+\dots+{a_N}^2\big) \log\log T}}\bigg) 
  =      e^{- \frac{\omega^2}{2}}N(T) + o(N(T)) . 
\]
The $o$-term constant depends at most on the $a_i, m_i$ and $N$.
\end{cor} 

\begin{proof}
We argue from the formula in  \eqref{eq:lemma fourier} of Lemma~\ref{fourier linear combination}. 
Let
\[
\mathcal{J}_0 = \int_0^1
         \exp(i \omega \Re{\CMcal{P}_{\mathcal{L}}(\underline{\theta})})\mathop{d\underline{\theta}},
\]
denote the integral in the first main term in \eqref{eq:lemma fourier}.
Recall (see \eqref{P_L theta}) that
$$
\CMcal{P}_{\mathcal{L}}(\underline{\theta})
=\sum_{p\leq X^2}\frac{(a_1\chi_1(p)+
\cdots +a_N \chi_N(p))e^{2\pi i \theta_{p}}}{\sqrt p}.
$$
We write the $p^{th}$ coefficient  here as
$$
 \nu_p e^{2\pi i\b_p}= a_1\chi_1(p)+\cdots +a_N\chi_N(p) ,
 $$
 where
 $ \nu_p=|a_1\chi_1(p)+\cdots+a_N\chi_N(p)| \geq0$ and
 $0\leq \b_p <1$.  If $\nu_p=0,$ then we set $\b_p=0$. 
 Recalling the definition of the Bessel function $J_0(z)= \int_0^1 \exp\big( i z\cos(2\pi\theta)\big)\mathop{d\theta} $ and using the independence of the $\theta_p$, we see that
\be\notag
\begin{split}
\mathcal{J}_0 = & \prod_{p\leq X^2} \int_0^1\exp\Big( \frac{i \omega\nu_p}{\sqrt{p}}\cos(2\pi(\beta_p+\theta_p))\Big) d\theta_p\\
=& \prod_{p\leq X^2} J_0 \Big( \frac{\omega \nu_p}{\sqrt{p}}\Big) . 
\end{split}
\ee
It is known that
\[
J_0(2z) = e^{-z^2} + O \big( |z|^4 e^{-z^2} \big) \quad \text{for} \quad |z| \leq 1.
\]
Thus, if $0\leq \omega \leq 2\sqrt{2}/\sum_{j=1}^{N}|a_j|$, then
$$
\mathcal{J}_0 =\prod_{p\leq X^2} \exp\Big(-\frac{ \omega^2 \nu_p^2}{4p}\Big)
\prod_{p\leq X^2}\Big(1+O\Big(\frac{\omega^4\nu_{p}^4}{p^2}  \Big)\Big)
= \exp\Big(-\frac{\omega^2}{4} \sum_{p\leq X^2}\frac{\nu_p^2}{p}\Big)\,(1+O(\omega^4)),
$$
where the $O$-term depends on the moduli $m_i$ of the $\chi_i$ and the $a_i$.
We set
$$
 \Psi_{\mathcal L}= \sum_{p\leq X^2}\frac{\nu_p^2}{p}
 $$
and note for later that
\be\label{Psi_L}
\begin{split}
 \Psi_{\mathcal L}= &\sum_{i, j=1}^{N}a_i a_j  \sum_{p\leq X^2}\frac{\chi_i(p)\overline{\chi_j(p)}}{p}
 = \sum_{j=1}^{N}a_j^2  \sum_{p\leq X^2}\frac{|\chi_i(p)|^2}{p}
 + \sum_{i\neq j=1}^{N}a_i a_j  \sum_{p\leq X^2}\frac{\chi_i(p)\overline{\chi_j(p)}}{p}\\
 =& \Psi(X^2) (a_1^2+ \cdots +a_N^2)+O(1),
\end{split}
\ee
where the  $O$-term again depends on  the $m_i$   and  $a_i$. 
Hence, we find that
\be\notag
\begin{split}
\mathcal{J}_0 =\exp\Big(- \frac{ \omega^2 \Psi_{\mathcal L} }{4}\Big)\, (1+O(\omega^4)),
\end{split}
\ee
and the first main term in  \eqref{eq:lemma fourier} equals
\be\label{first main}
\begin{split}
\exp\Big(- \frac{ \omega^2 \Psi_{\mathcal L} }{4}\Big) N(T)\, (1+O(\omega^4))
\end{split}
\ee
for $\omega \in [0, 2\sqrt{2}/\big(\sum_{j=1}^{N}|a_j|\big)]$.

Next, we bound the second main term in \eqref{eq:lemma fourier}. We set
    \[
     {\mathcal J}_\ell
     =  \int_0^1\exp\bigs( i\omega\Re{\CMcal{P}_{\mathcal L}(\underline{\theta})}\bigs)
      e^{2\pi i\ell\theta_q}\mathop{d\underline{\theta}} \, ,
    \]
so that the second integral in \eqref{eq:lemma fourier} can be written as
    \begin{align*}
     \int_0^1 \exp\bigs(i\omega\Re{\CMcal{P}_{\mathcal L}(\underline{\theta})}\bigs) 
      \Re(e^{2\pi i\ell\theta_q})\mathop{d\underline{\theta}}       
    =   \frac12  \big(  {\mathcal J}_\ell +  {\mathcal J}_{-\ell} \big) .
    \end{align*}
We thus need to bound ${\mathcal J}_{\pm \ell}$. With $\nu_p$ and $\beta_p$ as above, and by  the independence 
of the $\theta_p$'s, we see that
        \be\label{J_l}
        \begin{split}
        {\mathcal J}_\ell
        =\int_0^1 \exp\Big( i &\Big(\omega\sdfrac{\nu_q\cos(2\pi(\theta_q+\b_q))}{\sqrt q}+2\pi\ell\theta_q\Big)\Big)\mathop{d{\theta_q}} \\
        \cdot \prod_{\substack{p\leq X^2\\ p\neq q}} &\int_0^1 
        \exp\Big(i\omega\sdfrac{\nu_p\cos(2\pi(\theta_p+\b_p))}{\sqrt p}\Big)\mathop{d{\theta_p}}.
        \end{split}
        \ee      
Here recall that, for $\ell $ a nonnegative integer, the Bessel function of  order $\ell$ is  
    \be\label{Jell integral}
    J_\ell(z)
    =(-i)^\ell\int_0^1 \exp\Big(i \big(z\cos(2\pi\theta)+2\pi  \ell\theta\big)\Big)\mathop{d\theta}.
    \ee          
After the change of variable  $\theta_p \to \theta_p-\b_p$, we see that each factor in \eqref{J_l} corresponding to a prime $p$ other than $q$ is equal to
     \[
    J_0\Big(\frac{\omega \nu_p}{\sqrt p}\Big).
     \]
For the integral in \eqref{J_l}  corresponding to the prime $q$, a similar change of variable leads to
       \[
        \int_{-\b_q}^{1-\b_q}
         \exp\Big( i\Big(\omega \nu_q\frac{\cos(2\pi\theta_q)}{\sqrt q}+2\pi\ell\big(\theta_q-\b_q\big)\Big)\Big)
        \mathop{d{\theta_q}}.
        \]
Since the integrand has period $1$, it follows from \eqref{Jell integral} that this equals
        \[
         (ie^{-2\pi i \b_q})^\ell J_{\ell}\Big(\frac{ \omega\nu_q }{\sqrt q}\Big).
        \]     
 Inserting these   expressions   into \eqref{J_l}, we obtain for each $\ell$ with
 $1\leq \ell \leq K$
        \be\notag
        {\mathcal J}_\ell
        =(ie^{-2\pi i \b_q})^\ell J_{\ell}\Big(\frac{ \omega  \nu_q}{\sqrt q}\Big)
        \prod_{\substack{p\leq X^2,\\p\neq q}}J_0\Big(\frac{\omega \nu_p }{\sqrt p}\Big).
        \ee       
Since $J_{-\ell}(z)=(-1)^\ell J_{\ell}(z)$,   we also have
        \[
        {\mathcal J}_{-\ell}
        =(-ie^{-2\pi i \b_q})^\ell J_{\ell}\Big(\frac{ \omega \nu_q}{\sqrt q}\Big)
        \prod_{\substack{p\leq X^2,\\p\neq q}}J_0\Big(\frac{ \omega \nu_p}{\sqrt p}\Big).
        \]
From the bound $\nu_p \leq |a_1|+\cdots+|a_N|$ and the argument in \cite[pp. 6027-28]{C} based on standard estimates for Bessel functions, we find that if $\omega \in[0, (\log\log T)^2]$, then
       \be\notag
       \begin{split}
        {\mathcal J}_{\ell}
       \ll &\frac{\big((|a_1|+\cdots+|a_N|)\, \omega \big)^{|\ell|}}{|\ell| !q^{\frac{|\ell|}{2}}}  e^{-c\Psi_{\mathcal{L}}\omega^2+\sqrt{2}
       (|a_1|+\cdots+|a_N|) \, \omega}  \\
       \ll &\frac{((|a_1|+\cdots+|a_N|)\omega)^{|\ell|}}{|\ell| !q^{\frac{|\ell|}{2}}}  e^{ -c\Psi_{\mathcal{L}}\omega^2}.
       \end{split}
        \ee        
Using this bound and \eqref{K X}, we obtain
\begin{align*}
-\frac{T}{\pi}\sum_{q\leq X^2} \sum_{1\leq\ell\leq K} \frac{\log q}{q^{\frac{\ell}{2}}}
        \int_0^1\exp( i\omega\Re{\CMcal{P}_{\mathcal L}(\underline{\theta})})
        \Re( e^{ i\ell\theta_q}) \mathop{d\underline{\theta}} 
\ll
T\log X =o(N(T)).
\end{align*}
Combining this with \eqref{first main} and observing that the $O$-term in \eqref{eq:lemma fourier} is $o(N(T))$, we find that 
for $\omega \in[0, 2\sqrt{2}/\sum_{j=1}^{N}|a_j|]$,
 \be\notag
        \begin{split}
        \sum_{0 < \g \leq T}\exp( i \omega&\Re{\CMcal{P}_{\mathcal L}(\g)}) 
        = \exp\Big(- \frac{ \omega^2 \Psi_{\mathcal L} }{4}\Big) N(T)\, (1+O(\omega^4)) +o(N(T)),
\end{split}
\ee
where the constant implicit in the $o$-term depends only on the $a_i$ and $m_i$. 
Now we replace $\omega$ here by $\omega/\sqrt{\tfrac{1}{2}\big( {a_1}^2+\dots+{a_N}^2\big) \log\log T}$ and obtain
 \be\label{sum gamma 2}
        \begin{split}
        \sum_{0 < \g \leq T}\exp\Big( i \omega&\frac{\Re{\CMcal{P}_{\mathcal L}(\g)} }{\sqrt{\tfrac{1}{2}\big( {a_1}^2+\dots+{a_N}^2\big) \log\log T}}\Big) \\
        =& \exp\Big(- \frac{ \omega^2 \Psi_{\mathcal L} }{2 \big( {a_1}^2+\dots+{a_N}^2\big) \log\log T}\Big) N(T)\, 
        \Big(1+O \Big(\frac{\omega^4}{(\log\log T)^2} \Big) \Big) +o(N(T)),
\end{split}
\ee
By \eqref{Psi_L} and our choice of $X$ in \eqref{K X},
$$
 \Psi_{\mathcal L}=(a_1^2+\cdots +a_N^2)\Psi(X^2) 
        =(a_1^2+\cdots+a_N^2)\log\log T+O(\log\log\log T).
$$
Thus, the exponential on the right-hand side of \eqref{sum gamma 2} equals
 \be\notag
 \begin{split}
 \exp\Big(- \frac{ \omega^2 }{2}\Big)\Big(1+O\Big(\frac{\omega^2\log\log\log T}{\log\log T}\Big)\Big).
 \end{split}
 \ee
 It now follows that 
  \be\notag
        \begin{split}
        \sum_{0 < \g \leq T}\exp\Big( i \omega&\frac{\Re{\CMcal{P}_{\mathcal L}(\g)} }{\sqrt{\tfrac{1}{2}\big( {a_1}^2+\dots+{a_N}^2\big) \log\log T}}\Big) 
        = \exp\Big(- \frac{ \omega^2  }{2} \Big) N(T) +o(N(T)),
\end{split}
\ee
for 
$$\omega\in\Big[0, \tfrac{\sqrt{ {2} ( {a_1}^2+\dots+{a_N}^2) }}{\sum_{i=1}^N{|a_i|}}(\log\log T)^{\frac12-\e}\Big].$$ 
Since
$\sum_{i=1}^N|a_i | \leq \sqrt{N\sum_{i=1}^N a_i^2 }$, we see that $\omega$ may be restricted to the smaller interval
$[0, \sqrt{2/N}(\log\log T)^{\frac12-\e}]$. This completes the proof of the corollary.
 
 \end{proof}


\section{Completing the Proof of Theorem~\ref{distr of linear combination} }

In this section, we prove the following result, which connects the characteristic functions of 
$\Re \CMcal{P}_{\mathcal{L}}(\gamma)$ and $\mathcal{L}(\rho).$
From this, the claim of Theorem \ref{distr of linear combination} follows by L\'{e}vy's continuity theorem.

\begin{prop}\label{last prop}
For $|\omega| =o\Big(  \frac{\sqrt{\log\log T}}{(\log\log T)^{6}\log\log\log T}\Big)$ and $X = T^{\frac{1}{16 (\log\log T)^{6}}}$, we have
\begin{align*}
 \sum_{0<\g \leq T} \exp \bigg(i\omega  &\frac{\mathcal{L}(\rho)}{\sqrt{\tfrac{1}{2}\big( {a_1}^2+\dots+{a_N}^2\big) \log\log T}} \bigg)
\\
&=  \sum_{0<\g \leq T} 
\exp \bigg(  i\omega \frac{\Re \CMcal{P}_{\mathcal{L}}(\gamma)}{\sqrt{\tfrac{1}{2}\big( {a_1}^2+\dots+{a_N}^2\big) \log\log T}}    \bigg)
+o(N(T)) . 
\end{align*}
\end{prop}

\begin{proof}

Let $\chi_j$ denote one of the characters $\chi_1, \dots, \chi_N$. For each $0<\g \leq T$, we have that either 
$\eta_{\chi_j}(\g)=\min_{\g_{\chi_j}} |\g-\g_{\chi_j}|  \leq \frac{1}{\log T \log\log T}$ 
or $\eta_{\chi_j}(\g) > \frac{1}{\log T \log\log T}$. 
By Hypothesis $\mathscr{H}_{0, \chi_j}$, we know that
\begin{equation}\label{eq: no of gamma star}
\sum_{\substack{0 < \g \leq T, \\ \eta_{\chi_j}(\g) \leq \frac{1}{\log T \log\log T}  }} 1 = o(N(T)). 
\end{equation}
for each $j$. We let $\g^\ast$ denote any $\gamma$ such that $ \eta_{\chi_j}(\g) > \frac{1}{\log T \log\log T} $ for all $\chi_j$. 
By \eqref{eq: no of gamma star},  the set of  $\gamma$'s not satisfying this has cardinality $o(N(T))$.  Thus,  we may restrict the sums over the $\gamma$'s in our proposition  to $\gamma^\ast$'s at the cost of an error term of size $o(N(T))$.

Now, by the inequality
\[
\big| e^{ix} - e^{iy}\big| = \Big|  \int_x^y e^{iu} \mathop{du} \Big| \leq |x-y|
\]
for $x$ and $y$ real, and by Lemma \ref{log L approximate formula}, it suffices to bound 
\be\label{remainder bd}
\frac{|\omega|}{\sqrt{\log\log T}}\sum_{k=1}^{4}  \sum_{0<\g^\ast \leq T}  r_{\chi_j , k}(X, \g^\ast)
        +O\Big(\frac{|\omega|}{\sqrt{\log\log T}}  \sum_{0<\g^\ast \leq T}  \frac{  \log(m_j |\g^*|)}{\log X}\Big)
\ee
for each $\chi_j$, where  
        \be\notag
        \begin{split}
        r_{\chi_j, 1}(X, \g^\ast)&=\Big|\sum_{p\leq X^2}\frac{(1-w_X(p))\chi_j(p)}{p^{\frac12+i\g^\ast}}\Big|, \qquad
        r_{\chi_j, 2}(X, \g^\ast)=\Big|\sum_{p\leq X}\frac{w_X(p^2)\chi_j(p^2)}{p^{1+2i\g^\ast}}\Big|, \\
        &r_{\chi_j, 3}(X, \g^\ast)=\frac{1}{\log X} \int_{1/2}^{\infty} X^{\frac 1 2-\s}\Big|\sum_{p\leq X^2}\frac{\Lambda_X(p)\chi_j(p) \log(Xp)}{p^{\s+i\g^\ast}}\Big|\mathop{d\s},
        \\  
        &\qquad r_{\chi_j, 4}(X, \g^\ast)=\Big(1+\log^+\frac{1}{\eta_{\chi_j}(\g^\ast)\log X}\Big)\frac{E_{\chi_j}(X, \g^\ast)}{\log X},
        \end{split}
        \ee
and
$$
        E_{\chi_j}(X, \g^\ast) =\Big|\sum_{n\leq X^2} \sdfrac{\Lambda_X(n)\chi_j(n)}{n^{\s_1+i\g^\ast}}\Big|+\log(m_j|\g^\ast|). 
$$

For $r_{\chi_j, 1}(X, \g^\ast)$ observe that $(1-w_X(p))\chi_j(p) \ll \frac{\log p}{\log X}$. Thus,
\[
\sum_{0<\g^\ast \leq T}  r_{\chi_j, 1}(X, \g^\ast) \ll
\sqrt{N(T)} \bigg( \sum_{0<\g^\ast \leq T}  r_{\chi_j, 1}(X, \g^\ast)^2\bigg)^{\frac12}
\ll N(T),
\]
where we have applied Lemma \ref{Dir poly mmt} in an obvious way. By the same lemma, since $w_X(p^2) \leq 1$, it follows that $
\sum_{0<\g^\ast \leq T}  r_{\chi_j , 2}(X, \g^\ast) 
\ll N(T)
$.
The same bound can be obtained for the average of $r_{\chi_j , 3}$ by the proof of Proposition 4.3 in~\cite{C}. It remains to bound the average of $ r_{\chi_j , 4}$. By our choice of $X$ and the definition of $\gamma^\ast$, 
\[
\sum_{0<\g^\ast \leq T} r_{\chi_j , 4}(X, \g^\ast)
\ll \frac{1}{\log X} \log\bigg(\frac{\log T \log\log{T}}{\log X}\bigg) 
 \sum_{0<\g^\ast \leq T} \bigg(
\Big|\sum_{n\leq X^2} \sdfrac{\Lambda_X(n)\chi_j(n)}{n^{\s_1+i\g^\ast}}\Big|+\log(m_j|\g^\ast|)
\bigg).
\] 
By Cauchy's inequality and then Lemma \ref{Dir poly mmt},
\begin{align*}
&\sum_{0<\g^\ast \leq T} \bigg(
\Big|\sum_{n\leq X^2} \sdfrac{\Lambda_X(n)\chi_j(n)}{n^{\s_1+i\g^\ast}}\Big|+\log(m_j|\g^\ast|)
\bigg)
\\
\ll 
& \sqrt{N(T)} \bigg(\sum_{0<\g^\ast \leq T} \Big(
\Big|\sum_{n\leq X^2} \sdfrac{\Lambda_X(n)\chi_j(n)}{n^{\s_1+i\g^\ast}}\Big|+\log(m_j|\g^\ast|)
\Big)^2 \bigg)^{1/2} \ll N(T) \log T .
\end{align*}
Therefore,
\[
\sum_{0<\g^\ast \leq T} r_{\chi_j , 4}(X, \g^\ast)
\ll N(T) \log\log\log{T} \frac{\log T}{\log X}\ll N(T)(\log\log T)^{6} \log\log\log T. 
\]
Hence, \eqref{remainder bd} is
$$
\ll  \frac{|\omega|}{\sqrt{\log\log T}} N(T) (\log\log T)^{6} \log\log\log T.
$$
This is $o(N(T))$ provided that $$|\omega| =o\Big(  \frac{\sqrt{\log\log T}}{(\log\log T)^{6}\log\log\log T}\Big).$$
This completes the proof of Proposition~\ref{last prop} as well as Theorem~\ref{distr of linear combination}.

\end{proof}

\section{Proof of Theorem~\ref{n-variate}}

We use an argument similar to the one in~\cite{Hsu Wong}.  
First, it is well known that a  vector $\mathbf X=(X_1, \dots, X_N)$ of real-valued random variables has an $N$-variate normal distribution (see Section~\ref{intro} for the definition) if and only if every linear combination $a_1X_1+a_2X_2+\cdots+a_NX_N$, with each $a_j\in \mathbb R$, has a univariate normal distribution. (If the $a_j$ are all zero, we  regard the sum as being normally distributed with mean $0$ and variance $0$.) 
Now, by Theorem~\ref{distr of linear combination}, 
\be\label{lin comb}
a_1 \frac{ \log|L(\r, \chi_1)|}{\sqrt {\tfrac12 \log\log T}}+a_2 \frac{ \log|L(\r, \chi_2)|}{\sqrt {\tfrac12 \log\log T}} + \dots + a_N\frac{\log|L(\r, \chi_N)|}{\sqrt {\tfrac12 \log\log T}} 
\ee
 is approximately normal with mean $0$ and variance $a_1^2+\cdots+a_N^2$.
 Thus, 
\be\notag
\Bigg(
 \frac{ \log|L(\r, \chi_1)|}{\sqrt {\tfrac12 \log\log T}}, \ \dots \ ,\frac{\log|L(\r, \chi_N)|}{\sqrt {\tfrac12 \log\log T}} \Bigg)
\ee
has approximately an $N$-variate normal distribution with mean vector $\mathbf 0_N$. 

Next, we determine its covariance matrix $\mathbf C$. Abusing notation slightly by writing $X_i$ for $\log|L(\r, \chi_i)|/{\sqrt {\tfrac12 \log\log T}}$, we see that $\mathbf C$ has entries  
 $\sigma_{i, j}=\mathrm{Cov}(X_i, X_j)$ when $i\neq j$, and $\mathrm{Var}(X_i)$ when $i=j$.
 Taking $a_i =1$ and all the other $a_j=0$ in \eqref{lin comb}, we see that $\sigma_{i,i}=1$. Furthermore, 
 for any $i\neq j$, if we take $a_i=a_j=1$ and the other $a_k$ to be $0$ in \eqref{lin comb}, we then see that 
 $\mathrm{Var}(X_i+X_j)=1+1=\mathrm{Var}(X_i)+\mathrm{Var}(X_j)$. On the other hand, one has in general that $\mathrm{Var}(X_i+X_j) =\mathrm{Var}(X_i)+\mathrm{Var}(X_j) + 2\mathrm{Cov}(X_i, X_j)$.  Thus
 $\s_{i, j} =0$. It now follows that $\mathbf C =\mathbf I_N$. This proves the first assertion of Theorem~\ref{n-variate}. The second, that the $X_i$ are approximately mutually independent, follows from the fact that the covariance matrix is diagonal.

 
\section{Proof of Theorem~\ref{main thm 3}}  

We begin by proving  the following lemma.
  
\begin{lem}\label{lem 9}
Let $\chi_i$ and $\chi_j$ be distinct primitive Dirichlet characters. Then for all $\g\in (0,T]$ except possibly for a subset of cardinality at most $o(N(T))$, we have
    \[
   \big| \log|L(\r,\chi_i)|-\log|L(\r,\chi_j)| \big| \, > (\log\log T)^{1/4}.
    \]
Similarly, for all $\g\in (0,T]$ except possibly for a subset of cardinality at most  $o(N(T))$, we have
    \[
   \big| \log|L(\r,\chi_i) | \big| \, > (\log\log T)^{1/4}.
    \]
   
\end{lem}

\begin{proof}
By Theorem \ref{distr of linear combination}, the sequence 
$$
   \log|L(\r,\chi_i)|-\log|L(\r,\chi_j)| ,   \quad            0<\g\leq T,
$$
is approximately normal  with mean $0$ and variance $\log\log T$. Thus,
    \begin{align*}
    \#\bigs\{0<\g \leq T :  \   \big| \log|L(\r,\chi_i)|&-\log|L(\r,\chi_j)|  \big|\leq (\log\log T)^{\frac14}\bigs\}  \\
 = &\frac{1}{\sqrt{2\pi}}N(T) \int_{-(\log\log T)^{-1/4}}^{(\log\log T)^{-1/4}} e^{-\frac{x^2}{2}} \mathop{dx} +o(N(T)) \\
    = &o(N(T)),
      \end{align*}
which was to be shown.  The proof of the second assertion is similar.
\end{proof}  


We are now ready to count solutions of 
   \[
   F(\rho)=  c_1L(\rho, \chi_1)+c_2 L(\rho, \chi_2)+\dots+c_N L(\rho, \chi_N)=a.
   \] 
For each pair of indices $1\leq i\neq  j  \leq N$, let
$$
\mathscr A_{i, j} =   \bigs\{ 0<\g \leq T : \big| \log|L(\r,\chi_i)|-\log|L(\r,\chi_j)| \big| \, \leq (\log\log T)^{1/4} \bigs\}.
$$
Notice that $\mathscr A_{i,j}=\mathscr A_{j,i}.$
By Lemma~\ref{lem 9},  
$$
\# \mathscr A_{i, j}=o(N(T)).
$$ 
Hence the set
$$
\mathscr A= \bigcup_{1\leq i<j \leq N} \mathscr A_{i, j}
$$
 also has cardinality $o(N(T))$. Similarly, by Lemma~\ref{lem 9},  if 
$$
\mathscr B_{i} =   \bigs\{ 0<\g \leq T : \big| \log|L(\r,\chi_i)|  \big| \, \leq (\log\log T)^{1/4} \bigs\},
$$
then
$$
\mathscr B=\bigcup_{1\leq i  \leq N} \mathscr B_{i},
$$
has cardinality  $o(N(T)).$

For a given  zero $\r$ with $0<\g \leq T$ and $\r\notin \mathscr  A \cup  \mathscr B$,  
let $i_0$ be the (unique) index  for which
$$
 |L(\rho, \chi_{i_0})|= \max_{1\leq i \leq N} |L(\rho, \chi_i)|. 
$$
By Hypothesis $\mathscr D$, none of our $L$-functions vanishes at $\r$. Furthermore, no $c_i$ equals zero. Hence, we may write
    \begin{equation}\label{eq: F rho}
    F(\rho)= c_{i_0} L(\rho, \chi_{i_0}) \bigg(1+\sum_{i\neq i_0} \frac{c_i}{c_{i_0}} \frac{L(\rho, \chi_{i})}{L(\rho, \chi_{i_0})} \bigg).
    \end{equation}
By Lemma~\ref{lem 9}, we must have
\[ 
\log \Big|  \frac{L(\rho, \chi_{i})}{L(\rho, \chi_{i_0})} \Big| < -(\log\log T)^{\frac14}.
\]
Thus, for each $i\neq i_0$,
\[
\big|L(\rho, \chi_{i})\big| \leq e^{-(\log\log T)^{1/4}} \big|L(\rho, \chi_{i_0})\big|.
\]
It follows that
$$
\bigg| \sum_{i\neq i_0} \frac{c_i}{c_{i_0}} \frac{L(\rho, \chi_{i})}{L(\rho, \chi_{i_0})} \bigg|
\ll e^{-(\log\log T)^{1/4}}.
$$
Thus, by \eqref{eq: F rho}, we see that
\be\label {F} 
|F(\rho)|
=  | c_{i_0} L(\rho, \chi_{i_0})  |  \  \big(1+ O ( e^{-(\log\log T)^{1/4}} )\big).
\ee
Since $\r \notin B$, we have either that
$$
| L(\rho, \chi_{i_0}) | < e^{-(\log\log T)^{1/4}} 
\qquad \hbox{or} \qquad
|L(\rho, \chi_{i_0}) | > e^{(\log\log T)^{1/4}} .
$$
In either case,  for $T$ sufficiently large, $F(\rho)\neq a$ by  \eqref{F}.   
 This holds for all but $o (N(T))$ values of $\r$ with $0<\g \leq T$, so the proof of the theorem is complete.



\end{document}